\documentclass[a4paper,12pt]{article}

\usepackage{amssymb}
\usepackage{amsthm}
\usepackage{amsmath}
\usepackage{graphicx}
\usepackage{psfrag}
\usepackage{array,hhline}

\newtheorem{lemma}{Lemma}[section]
\newtheorem{theorem}[lemma]{Theorem}
\newtheorem{proposition}[lemma]{Proposition}
\newtheorem{conjecture}[lemma]{Conjecture}
\newtheorem{corollary}[lemma]{Corollary}
\theoremstyle{definition}
\newtheorem{definition}[lemma]{Definition}

\numberwithin{equation}{section}
\numberwithin{figure}{section}

\newcommand{\wti}[1]{{\widetilde #1}}

\newcommand{\nix}{{\mbox{\vphantom x}}}

\newcommand{\Aset}{\mathcal{A}}
\newcommand{\Bset}{\mathcal{B}}
\newcommand{\Cset}{\mathcal{C}}

\newcommand{\Lset}{\mathcal{L}}
\newcommand{\Nset}{\mathcal{N}}

\newcommand{\Sset}{\mathcal{S}}

\newcommand{\Xset}{\mathcal{X}}
\newcommand{\Yset}{\mathcal{Y}}

\begin{document}

\title{\Large{On the Gasca-Maeztu conjecture for $n=6$}}

\author{H. Hakopian, \  G. Vardanyan, \  N. Vardanyan}

\date{}

\maketitle

\begin{abstract}
A two-dimensional $n$-correct set is a set of nodes admitting unique
bivariate interpolation with polynomials of total degree at most ~$n$.
We are interested in correct sets with the property that all fundamental
polynomials are products of linear factors. In 1982, M.~Gasca and
J.~I.~Maeztu conjectured that any such set necessarily contains $n+1$
collinear nodes. So far, this had only been confirmed for $n\leq 5.$ In this paper, we make a step for proving the case $n=6.$ 
\end{abstract}




\section{Introduction}
Denote by $\Pi_n$ the space of bivariate polynomials of total degree
$\le n,$ for which

$
\qquad\qquad\qquad N:=N_n:=\dim \Pi_n=(1/2){(n+1)(n+2)}.
$

\noindent Let
$\Xset:=\Xset_s=\{ (x_1, y_1), \dots , (x_s, y_s) \}$  be a set of $s$ distinct nodes in the plane.

The problem of finding a polynomial $p \in \Pi_n$ satisfying the conditions
\begin{equation}\label{int cond}
p(x_i, y_i) = c_i, \ \ \quad i = 1, 2, \dots s  ,
\end{equation}
for a data $\bar c:=\{c_1, \dots, c_s\}$ is called \emph{interpolation problem.}
\begin{definition}
A set of nodes $\Xset_s$ is called
$n$-\emph{correct} if for any data $\bar c$ there exists a
unique polynomial $p \in \Pi_n$, satisfying the conditions
\eqref{int cond}.
\end{definition}
A necessary condition of
$n$-correctness is: $\#\Xset_s=s = N.$

Denote by $p|_\Xset$ the restriction of $p$ on $\Xset.$
\begin{proposition} \label{correctii}
A set of nodes $\Xset$ with $\#\Xset=N$ is $n$-correct if and only if
$$p \in \Pi_n,\ p|_\Xset=0\implies p = 0.$$
\end{proposition}
A polynomial $p \in \Pi_n$ is called an $n$-\emph{fundamental polynomial}
for $ A \in \Xset$ if
$$ p|_{\Xset\setminus\{A\}}=0\ \hbox{and}\ p(A)=1.$$
We denote an
$n$-fundamental polynomial of $A \in\Xset$ by $p_A^\star=p_{A,\Xset}^\star.$

\begin{definition}
A set of nodes $\Xset$ is called $n$-\emph{independent} if each node
has $n$-fundamental polynomial. Otherwise, it is $n$-\emph{dependent.}
A set $\Xset$ is called \emph{essentially $n$-dependent} if none of its nodes
has $n$-fundamental polynomial.
\end{definition}
Fundamental polynomials are linearly independent. Therefore a
necessary condition of $n$-independence is $\#\Xset_s=s \le N.$

One can readily verify that a node set $\Xset_s$
is $n$-independent if and only if the interpolation problem
\eqref{int cond} is \emph{solvable}, i.e., for any data $\{c_1, \dots , c_s
\}$ there is a (possibly not unique) polynomial $p \in \Pi_n$
satisfying \eqref{int cond}.

A \emph{plane algebraic curve} is the zero set of some bivariate polynomial of degree $\ge 1.$~To simplify notation, we shall use the same letter,  say $p$,
to denote the polynomial $p$ and the curve given by the equation $p(x,y)=0$.
In particular, by $\ell$ (or $\alpha$) we denote a linear
polynomial from $\Pi_1$ and the line defined by the equation
$\ell(x, y)=0.$
\begin{definition} Let $\Xset$ be a set of nodes.
We say, that a line $\ell$ is a $k$-\emph{node line} if it passes through exactly $k$ nodes of $\Xset.$\end{definition}
The following proposition is well-known (see e.g. \cite{HJZ09b}
Prop. 1.3):
\begin{proposition}\label{prp:n+1ell}
Suppose that a polynomial $p \in
\Pi_n$ vanishes at $n+1$ points of a line $\ell.$ Then we have that
$
p = \ell  q  ,\ \text{where} \ q\in\Pi_{n-1}.
$
\end{proposition}
This implies that at most $n+1$ nodes of an $n$-independent set  can be collinear. An $(n+1)$-node line $\ell$ is called a \emph{maximal line} (C. de Boor, \cite{dB07}).

Set $$d(n, k) := N_n - N_{n-k} = (1/2) k(2n+3-k).$$

The following is a  generalization of Proposition \ref{prp:n+1ell}.
\begin{proposition}[\cite{Raf}, Prop. 3.1]\label{maxcurve}
Let $q$ be an algebraic curve of degree $k \le n$ with no multiple components. Then the following hold:

$(i)$ any subset of $q$ containing more than $d(n,k)$ nodes is
$n$-dependent;

$(ii)$ any subset ${\mathcal X}$ of $q$ containing exactly $d(n,k)$ nodes is $n$-independent if and only if
\vspace{-.5cm}

$$\quad p\in {\Pi_{n}}\ \text{and}\ \ p|_{{\mathcal X}} = 0 \implies  p = qr,\ \hbox{where}\ r \in \Pi_{n-k}.$$
\end{proposition}
\noindent Thus at most $d(n,k)$ $n$-independent nodes lie in a curve $q$ of degree $k \le n$.
\begin{definition}\label{def:maximal}
Let $\Xset$ be an $n$-independent set of nodes with $\#\mathcal X\ge d(n,k).$ A curve of degree $k \le n$ passing through $d(n,k)$ points
of $\mathcal X$ is called maximal.
\end{definition}
The following is a characterization of the maximal curves:
\begin{proposition}[\cite{Raf}, Prop. 3.3] \label{maxcor}
Let $\Xset$ be an $n$-independent set of nodes with with $\#\mathcal X\ge d(n,k).$ Then
a curve $\mu$ of degree $k,\ k\le n,$ is a maximal curve if and only if 
\vspace{-.5cm}

$$p\in\Pi_n,\ p|_{\mathcal X \cap\mu}=0 \implies p=\mu s,\ s\in\Pi_{n-k}.$$
\end{proposition}One readily gets from here that for a $GC_n$ set $\Xset$ and $\mu\in\Pi_k:$
\begin{equation}\label{maxmax}\mu \hbox{ is a maximal curve }  \iff \Xset\setminus \mu  \hbox{ is a $GC_{n-k}$ set.}\end{equation} 
In the sequel we will need the following results:
\begin{theorem}[case i=1:  \cite{HakTor}, Thm.~4.2; case i=2: \cite{HK}, Thm.~3] \label{th:-1+1}
Let $i=1$ or $2.$ Assume that ${\mathcal X}$ is an $n$-independent set of $d(n, k-i)+i$ nodes with $1+i\le k\le n-1.$ Then at most $2i$ different curves of degree
$\le k$ may pass through all the nodes of ${\mathcal X}.$

Moreover, there are such $2i$ curves for the set ${\mathcal X}$ if and only if all the nodes of ${\mathcal X}$ but $i$ lie in a maximal curve of degree $k-i.$
\end{theorem}

\begin{theorem}[\cite{HKV}, Thm.~2.5]\label{th:-2+3}
Assume that ${\mathcal X}$ is an $n$-independent set of $d(n, k-2)+3$ nodes with $3\le k\le n-2.$ Then at most three linearly independent curves of degree
$\le k$ may pass through all the nodes of ${\mathcal X}.$

 Moreover, there are such three curves for the set ${\mathcal X}$ if and only if all the nodes of ${\mathcal X}$  lie in a  curve of degree $k-1,$ or all the nodes of  ${\mathcal X}$ but three lie in a (maximal) curve of degree $k-2.$
\end{theorem}
Below we bring a characterization of $n$-dependent sets $\Xset$ with $\#\Xset\le 3n.$
\begin{theorem}[\cite{HM}, Thm.~5.1]\label{th:3n}
A set $\Xset$ consisting of at most $3\,n$ nodes is $n$-dependent if
and only if one of the following conditions holds.
\vspace{-2mm}
\begin{enumerate}
\setlength{\itemsep}{0mm}
\item
$n+2$ nodes are collinear,
\item
$2n+2$ nodes belong to a (possibly reducible) conic,
\item
$\#\Xset = 3n,$ and there exist $\gamma\in\Pi_3$ and $\sigma\in\Pi_n$
such that $\Xset = \gamma \cap \sigma$\,.
\end{enumerate}
\end{theorem}

\begin{corollary}\label{cor:3n-1}
A set $\Xset$ consisting of at most $3n-1$ nodes is $n$-dependent if
and only if either $n+2$ nodes are collinear, or
$2n+2$ nodes belong to a (possibly reducible) conic.
\end{corollary}

Consider special $n$-correct sets: $GC_n$ sets,  defined by Chung and Yao:
\begin{definition}[\cite{CY77}]
An n-correct set $\Xset$ is called $GC_n$ \emph{set}, if the $n$-fundamental polynomial of each node $A\in\Xset$ is a product of
$n$ linear factors.
\end{definition}

Now we are in a position to present the Gasca-Maeztu, or briefly GM
\begin{conjecture}[\cite{GM82}, 1982] \label{GM}
Any $GC_n$ set contains $n+1$ collinear nodes.
\end{conjecture}

So far, the GM conjecture has been confirmed to be true only for $n \leq 5$. The case $n=2$ is trivial. The case $n=3$ was established by M. Gasca and J. I. Maeztu in \cite{GM82}. The case
$n = 4$ was proved by J. R.  Busch \cite{B90}. Other proofs of this case have been published since then (see e.g. \cite{CG01}, \cite{HJZ09b}). The case $n=5$ was proved by H. Hakopian, K. Jetter and G. Zimmermann \cite{HJZ14}. Recently G. Vardanyan provided a simpler and shorter proof for this case \cite{GV2}.

In this paper we make a step in proving the Gasca-Maeztu conjecture for $n=6$ (see Proposition \ref{thm:main3}). The analogue of this step was crucial in the proof of the case  $n=5$ (see \cite{HJZ14}, Prop. 3.12; \cite{GV2}, Prop. 2.8).

\begin{definition} Let $\Xset$ be an $n$-correct set.
We say, that a node $A\in\Xset$
\emph{uses a line} $\ell$, if  $p_{A}^\star=\ell q,\quad q\in\Pi_{n-1}.$
\end{definition}

Since the fundamental polynomial in an $n$-correct set is unique we
get
\begin{lemma} \label{lm}
Suppose $\Xset$ is an $n$-correct set and a node  $A\in \Xset$ uses a line $\ell.$ Then $\ell$ passes through at least two nodes from $\Xset$, at which $q$ from the above definition does not vanish.
 \end{lemma}

\begin{definition} \label{def:Nll}
For a given set of lines $\ell_1,\dots,\ell_k$\,, we define
$\Nset_{\ell_1,\dots,\ell_k}$
to be the set of those nodes in $\Xset$ which do not lie in any of the
lines $\ell_i$, and for which at least one of the lines is not used.
\end{definition}
\noindent In the case of one line $\ell$ we have
$$ \Nset_\ell = \left\{ A\in\Xset : A \notin \ell, \text{ and $A$ is
  not using $\ell$}\right\}.$$
\begin{proposition}[\cite{HJZ09b}, Thm. 3.2]\label{prp:kell}
Assume that $\Xset$ is a $GC_n$ set, and
$\ell_1,\dots,\ell_k$ are lines. Then the following hold for $\Nset =
\Nset_{\ell_1,\dots,\ell_k}$\,.
\vspace{-2mm}
\begin{enumerate}
\setlength{\itemsep}{0mm}
\item
If $\, \Nset$ is nonempty, then it is essentially   $(n-k)$-dependent.
\item
$\Nset = \emptyset$ if and only if the product $\ell_1\cdots\ell_k$ is a maximal curve.
\end{enumerate}
\end{proposition}

\noindent For $k=1$ this result has been proved by Carnicer and Gasca \cite{CG01}.

Assume that $\Xset_i$ is a set of $k_i$ collinear points:
$$\Xset_i\subset \ell_i,\quad \#\Xset=k_i,\ i=1,2,3,\quad \ell_i \hbox{ is a line}.$$
Assume also that non of the points is an intersection point of the lines.

 Consider the set
$\Lset_{\Xset_1,\Xset_2,\Xset_3}$ of lines containing one point from
each of $\Xset_i$ $i=1,2,3,$ and denote by
$M_{k_1,k_2,k_3}$ the maximal possible number of such lines.

We shall need the following estimate (see \cite{HJZ09b}, \cite{HJZ14})
\begin{equation}\label{M332}M_{3,3,2}=5.\end{equation}


\subsection{The m-distribution sequence of a node}

In this section
we bring a number of concepts  from  \cite{HJZ14}, Section 2.

Suppose that $\mathcal X$ is a $GC_n$ set. Consider a node $A\in\Xset$ together with the set of $n$ used lines denoted by $\Lset_A.$
The $N-1$ nodes of $\Xset\setminus\{A\}$ belong to the lines of $\Lset_A.$
Let us order the lines of $\Lset_A$ in the following way:

The line $\ell_1$ is a line in $\Lset_A$ that passes through maximal number of nodes of $\Xset,$ denoted by $k_1:$  $\Xset\cap\ell_1=k_1.$
The line $\ell_2$ is a line in $\Lset_A$ that passes through maximal number of nodes of $\Xset\setminus\ell_1,$ denoted by $k_2:$  $(\Xset\setminus\ell_1)\cap\ell_2=k_2.$

In the general case the line $\ell_s,\ s=1,\ldots,n,$ is a line in $\Lset_A$ that passes through maximal number of nodes of the set $\Xset\setminus\cup_{i=1}^{s-1}\ell_i,$ denoted by $k_s:$  $(\Xset\setminus\cup_{i=1}^{s-1}\ell_i)\cap\ell_s=k_s.$
A correspondingly ordered  line sequence $$\Sset=(\ell_1,\ldots,\ell_n)$$
 is called  a \emph{maximal line sequence} or briefly an \emph{m-line sequence}  if the respective sequence $(k_1,\ldots,k_n)$ is the maximal in the lexicographic order \cite{HJZ14}. Then the latter sequence is called a \emph{maximal distribution sequence} or briefly an \emph{m-d sequence.}
Evidently, for the m-d sequence we have that
\begin{equation}\label{nor0}k_1 \geq k_2 \geq \cdots \geq k_n\ \hbox{and}\  k_1 + \cdots + k_n = N-1.\end{equation}

Though the m-distribution sequence for a node $A$ is unique, it may correspond to
several m-line sequences.

An intersection point of several lines of $\Lset_A$
is counted for the line containing it which appears in $\Sset$ first.
A node in $\Xset$ is called  \emph{primary} for the line it
is counted for, and  \emph{secondary} for the other lines
containing it.

\noindent According to Lemma~\ref{lm},
a used line contains at least two primary nodes:
\begin{equation}\label{eq:kgeq2}
  k_i \geq 2 \quad\text{for } i=1,\ldots,n \,.
\end{equation}

\noindent Let  $(\ell_1,\ldots,\ell_k)$ be a line sequence.
\begin{definition}
We say that a
polynomial has $(s_1,\ldots,s_k)$ \emph{primary zeroes} in the lines
$(\ell_1,\ldots,\ell_k)$ if the counted zeroes are primary nodes in the
respective lines.
\end{definition}
From Proposition \ref{prp:n+1ell} we get
\begin{corollary}\label{cor22} If $p\in\Pi_{m-1}$ has $(m,m-1,\ldots,m-k+1)$ primary zeroes in the lines
$(\ell_1,\ldots,\ell_{k})$ then we have that
$
p = \ell_1\cdots\ell_{k}  r  ,\ \text{where} \ r\in\Pi_{m-k-1}.
$
\end{corollary}

In some cases a particular line $\wti\ell$ used by a node is fixed and then
the properties of the other factors of the fundamental polynomial are studied.

In this case in the corresponding m-line sequence, called $\wti\ell$-\emph{m-line sequence}, one takes as the first line $\ell_1$ the line $\wti\ell,$ no matter through how many nodes it passes. Then the second and subsequent lines are chosen, as in the case of the  m-line sequence.
Thus the line $\ell_2$ is a line in $\Lset_A\setminus \{\wti\ell_1\}$ that passes through maximal number of nodes of $\Xset\setminus\wti\ell_1,$ and so on.

Correspondingly the   $\wti\ell$-\emph{m-distribution  sequence} is defined.

\section{The Gasca-Maeztu conjecture for $n=6$}

Now let us formulate the Gasca-Maeztu conjecture for
$n=6$ as:

\begin{theorem}\label{th:main}
Any $GC_6$ set contains seven collinear nodes.
\end{theorem}

To make a step for the proof assume by way of contradiction:

\noindent {\bf Assumption.} {\it The set $\Xset$ is a $GC_6$ set without a maximal line.}

In view of  \eqref{nor0} and
\eqref{eq:kgeq2} the only
possible m-distribution sequences for any node $A\in\Xset$ in the case $n=6$ with $N=28$ are
\begin{equation*}\label{eq:12cases}
\begin{matrix}
\text{(i)}\quad & (6,6,6,4,3,2) &\text{(ii)}\quad&(6,6,5,5,3,2) &\text{(iii)}\quad(6,6,5,4,4,2)\\ \text{(iv)}\quad &(6,6,5,4,3,3)
&\text{(v)}\quad &(6,6,4,4,4,3) &\text{(vi)}\quad(6,5,5,5,4,2)\\ \text{(vii)}\quad &(6,5,5,5,3,3)
&\text{(viii)}\quad &(6,5,5,4,4,3) &\text{(ix)}\quad(6,5,4,4,4,4)\\ \text{(x)}\quad &(5,5,5,5,5,2)
&\text{(xi)}\quad &(5,5,5,5,4,3) &\text{(xii)}\quad(5,5,5,4,4,4).
\end{matrix}\end{equation*}
Here we omitted the distribution sequences $(6,6,6,5,2,2)$ and $(6,6,6,3,3,3).$ The reason is that $\ell_1\ell_2\ell_3$ is a maximal cubic with $18\  (=6+6+6)$ nodes and, in view of \eqref{maxmax}, three $6$ must be followed by $4,3,2,$ as in above $(i).$

\section{Lines used several times}

\subsubsection*{A $2$-node line shared}

Consider a $2$-node line $\wti\ell.$ For the $\wti\ell$-m-distribution sequence of a node $A\notin \wti\ell$
there are only the following five possibilities:
\begin{equation}\label{eq:5cases}
\begin{matrix}
\text{(i)}    \  &(\wti 2,6,6,6,4,3) &\text{(ii)}   \   &(\wti 2,6,6,5,5,3) &\text{(iii)}  \   &(\wti 2,6,6,5,4,4)\\
\text{(vi)}   \   &(\wti 2,6,5,5,5,4)  &\text{(x)}    &(\wti 2,5,5,5,5,5).
\end{matrix}
\end{equation}
Note that in $\wti\ell$-m-d sequences, we use the tilde to indicate the place of $\wti\ell.$

It was proved in \cite{CG02}, Prop. 4.2, that any $2$-node line in a
$GC_n$ set $\Xset$ can be used at most by one node from $\Xset$.
This yields the following
\begin{proposition}\label{prp:2nl}
Assume that $\Xset$ is a $GC_6$-set, and suppose
that $\wti\ell$ is a $2$-node line.
Then $\wti\ell$ can be used by at most one node $A\in\Xset.$ The
m-d sequence of $A$ has to be one of $\text{(i), (ii), (iii), (vi),}$ and (x), presented in \eqref{eq:5cases}.
\end{proposition}

\subsubsection*{A $3$-node line shared}

Then, consider a $3$-node line $\wti\ell.$ For the $\wti\ell$-m-d sequence of a node $A\notin \wti\ell$ there are only the following possibilities:

\begin{equation*}\label{eq:7cases}
  \begin{matrix}
  \text{(i)}   &  (\wti 3,6,6,6,4,2)
  &\text{(ii)}  & (\wti 3,6,6,5,5,2)
    &\text{(iv)}  & (\wti 3,6,6,5,4,3)\\
       \text{(v)}   &  (\wti 3,6,6,4,4,4)
      &\text{(vii)}  & (\wti 3,6,5,5,5,3)
         &\text{(viii)} & (\wti 3,6,5,5,4,4)\\
   \text{(xi)} &\ (\wti 3,5,5,5,5,4).
  \end{matrix}
\end{equation*}

Here, and in all subsequent cases, denote
a respective $\wti\ell$-m-line sequence by
$(\wti\ell,\ell_2,\ldots,\ell_6).$ Denote also by $\ell_{AB}$ the line through the nodes $A$ and $B.$

Suppose that the line $\wti\ell$ is used by two
nodes $A$, $B\in \Xset$:
\begin{equation*}\label{eq:AB}
  p_A^\star = \wti\ell\,q_1
  \quad\text{and}\quad
  p_B^\star = \wti\ell\,q_2\,,\quad q_i\in \Pi_5 \,.
\end{equation*}

Then we have that the curves: $q_1,q_2\in\Pi_5,$ pass through  $6$-independent nodes of the set $\Yset:=\Xset\setminus (\wti\ell \cup \{A,B\}),\ \#\Yset=28-(3+2)=23.$

Note that $23=d(6,5-1)+1=d(6,4)+1=7+6+5+4+1.$

Therefore, in view of Theorem \ref{th:-1+1}, case i=1, we get that
all the nodes of $\Yset$ but one, denoted by $C,$  belong to a maximal curve $\mu$ of degree $4.$
Note that $ p_C^\star =\wti\ell\mu_4 \ell_{AB},$ meaning that the node $C$ uses $\wti\ell$ too.

Since $\Xset$ is a $GC$ set we conclude that $\mu$ has $4$ line-components coinciding with $\ell_2,\ldots,\ell_5$.
It is easily seen
 that these four lines have
$6,\ 6,\ 6,\ 4$ or $6,\ 6,\ 5,\ 5,$
nodes, respectively.
 For $D=A,B,C,$ we have that
 \begin{equation}\label{eq:ABC3}
  p_D^\star =\wti\ell\,\ell_2\cdots\ell_6,
\end{equation}
where $\ell_6$ is a line depending on $D$ with two primary nodes.

Thus the $\wti\ell$-m-d sequence indicated in \eqref{eq:ABC3} may correspond only to the m-d sequences (i) $(6,6,6,4,3,2)$  and (ii) $(6,6,5,5,3,2).$

Note that all the $6$  nodes in $\Xset\setminus \mu,$ included $C,$ share the $4$ line-components of $\mu.$ As it is proved in \cite{HakTor}, Corollary 6.1, no node in  $\mu$ uses the line $\wti\ell.$

Thus we have shown the following:

\begin{proposition}\label{prp:3nl}
Assume that $\Xset$ is a $GC_6$-set without a maximal line, and suppose
that a $3$-node line $\wti\ell$ is used by two nodes $A$, $B\in\Xset$. Then
there exists a
third node $C$ using $\wti\ell$ and $\wti\ell$ is used by exactly three nodes of $\Xset.$

Moreover, $A$, $B$, and $C,$
share four lines with either  $6,6,6,4,$ or $6,6,5,5,$ primary nodes, respectively.
Furthermore,
the m-d sequence of these three nodes is either $(6,6,6,4,\wti 3,2)$, or $(6,6,5,5,\wti 3,2)$, respectively.
\end{proposition}

\subsubsection*{A $4$-node line shared}

Now, consider a $4$-node line $\wti\ell.$
$\Xset$. The $\wti\ell$-m-d sequence of $A\notin \wti\ell$ has to be one of the following:
\begin{equation*}\label{eq:9cases}
  \begin{matrix}
  \text{(i)}   &  (\wti 4,6,6,6,3,2)
  & \text{(iii)} &  (\wti 4,6,6,5,4,2)
  &\text{(iv)}  &  (\wti 4,6,6,5,3,3)\\
   \text{(v)}   &  (\wti 4,6,6,4,4,3)
    &\text{(vi)}  &  (\wti 4,6,5,5,5,2)
      &\text{(viii)} &  (\wti 4,6,5,5,4,3)\\
     \text{(ix)}  &  (\wti 4,6,5,4,4,4)
  &\text{(xi)} &  (\wti 4,5,5,5,5,3)
   &\text{(xii)}   &  (\wti 4,5,5,5,4,4).
  \end{matrix}
\end{equation*}
Assume that the nodes $A,B,C\in\Xset$ use the line $\wti\ell.$
Then we have three curves: $p_A^\star,\ p_B^\star,\  p_C^\star,$ passing through $21=28-(4+3)$ $6$-independent nodes of the set $\Yset:=\Xset\setminus (\wti\ell \cup \{A,B,C\}).$

Note that $21=d(6,5-2)+3=d(6,3)+3=7+6+5+3.$

This, in view of Theorem \ref{th:-2+3}, implies that either

\noindent (a) all the nodes of $\Yset$ but three, i.e., $18$ nodes, belong to a maximal curve $\mu$ of degree $3,$ or

\noindent (b) all the nodes of $\Yset,$ i.e., $21$ nodes, belong to a curve $q$ of degree $4.$

Since any node outside of $\mu$ uses it we get that $\mu$ has $3$ line-components, passing through
$6+6+6$ nodes, respectively.

Concerning (b) note that $\wti\ell q$ is a maximal curve of degree $4$ and any node $D=A,B,C,$ uses $q:$
 \begin{equation*}\label{eq:ABC4}
  p_D^\star =\wti\ell q\ell_6,
\end{equation*}
where $\ell_6$ is a line depending on $D$ with two primary nodes.

Hence $q$ has $4$ line-components.
 It is easily seen that these four lines have either
$6+6+6+3,\ 6+6+5+4,$ or $6+5+5+5$
nodes, correspondingly. We readily get also that these lines coincide with the lines $\ell_2,\ldots,\ell_5,$ of the corresponding $\wti\ell$-m-distribution \eqref{eq:ABC3}.
Hence, these three cases may correspond only to the above cases $(i)$ and $(iii)$ and $(vi).$

Now suppose that except of $A,B,C,$ another node $D\in\Xset$ uses $\wti\ell.$
Then we have four curves passing through $20=28-(4+4)$ $6$-independent nodes.
We have that $20=d(6,5-2)+2=d(6,3)+2=7+6+5+2.$

Therefore, in view of Theorem \ref{th:-1+1}, case i=2, we obtain that all the nodes of $\Xset\setminus \{A,B,C,D\}$ but two, i.e., $18$ nodes belong to a maximal curve $\mu$ of degree $3.$ As was stated above this maximal curve has $3$ line-components with
$6+6+6$ nodes, correspondingly. We readily get also that these lines coincide with the lines $\ell_2,\ell_3,\ell_4.$
Consequently, this case may correspond only to the above case $(i).$  As it is proved in    \cite{HK}, Corollary, no node in  $\mu$ uses the line $\wti\ell.$

By summarizing we obtain the following
\begin{proposition}\label{prp:4nl3}
Assume that $\Xset$ is a $GC_6$-set without a maximal line, and suppose
that a $4$-node line $\wti\ell$ is used by three nodes $A$, $B$,
$C\in\Xset$. Then, $A$, $B$, and $C,$ besides $\wti\ell,$
share four lines with either  $6,6,6,3;\ $ $6,6,5,4;$ or $6,5,5,5,$ primary nodes, respectively.

Moreover the m-d sequence for $A$, $B$,
$C,$ is $(6,6,6,4,\wti 3,2),\ (6,6,5,5,\wti 3,2),\\ (6,6,5,4,\wti 4,2),$ or $(6,5,5,5,\wti 4,2) .$
\end{proposition}

\begin{proposition}\label{prp:4nl4}
Assume that $\Xset$ is a $GC_6$-set without a maximal line, and suppose
that some $4$-node line $\wti\ell$ is used by four nodes $A$, $B$,
$C,\ D\in\Xset$. Then, $\wti\ell$ is used by exactly $6$ nodes.

Moreover, besides $\wti\ell,$  these six nodes share also
three other lines  each passing through $6$ primary nodes.
Furthermore the m-d sequence for all six nodes is $(6,6,6,\wti 4,3,2).$

\end{proposition}


\subsubsection*{A $5$-node line shared}

Now suppose that $\wti\ell$ is a $5$-node line. The $\wti\ell$-m-d sequence of $A\notin\wti\ell$ has to be one of the following:
\begin{equation*}\label{eq:10cases}
  \begin{aligned}
   \text{(ii)}\  &  (\wti 5,6,6,5,3,2)
  &\text{(iii)}\ &  (\wti 5,6,6,4,4,2)
  &\text{(iv)}\  &  (\wti 5,6,6,4,3,3)\\
      \text{(vi)}\  &  (\wti 5,6,5,5,4,2)
   &\text{(vii)}\  &  (\wti 5,6,5,5,3,3)
     &\text{(viii)}\ &  (\wti 5,6,5,4,4,3)\\
     \text{(ix)}\  &  (\wti 5,6,4,4,4,4)
    &\text{(x)}\ &  (\wti 5,5,5,5,5,2)
  &\text{(xi)}\ &  (\wti 5,5,5,5,4,3)\\
   \text{(xii)}\   &  (\wti 5,5,5,4,4,4).
  \end{aligned}
\end{equation*}
Let us start with a well-known
\begin{lemma} \label{lem:mm-1}
Given $m$ linearly independent polynomials.
Then for any point $A$ there are $m-1$ linearly independent polynomials, in their linear span, vanishing at $A.$
\end{lemma}

\begin{proposition}\label{prp:5nl}
Assume that $\Xset$ is a $GC_6$-set without a maximal line, and
$\wti\ell$ is a $5$-node line used by five nodes of $\Xset.$ Then it  is used by exactly six nodes.

Moreover, besides $\wti\ell,$  these six nodes share also
three other lines  passing through $6,6,5$ primary nodes, respectively. Furthermore the m-d sequence for each of the six nodes is $(6,6,6,4,3,2),$ or $(6,6,5,5,3,2).$
\end{proposition}

\begin{proof} Assume that the nodes of the set $\mathcal A_5:=\{A_1,\ldots, A_5\}\subset\Xset$ use the line $\wti\ell.$
Assume that
$$p^\star_{A_1}=\wti\ell \ell_2\cdots \ell_6.$$
Evidently,  the nodes $A_2,\ldots, A_5$ belong to the lines $\ell_2,\ldots, \ell_6.$

In view of Lemma \ref{lem:mm-1} for any points $T_i,\ i=1,2,3,$ there is a polynomial
$$p_0\in \mathcal P_4:= linear span \{p^\star_{A_2},\ldots,p^\star_{A_5}\},\quad p_0\neq 0,$$

\noindent such that $p_0(T_i)=0,\ i=1,\ldots,3.$
On the other hand we have that
$$p_0=\wti\ell q_0,\quad q_0\in\Pi_5.$$
Assume that the three points are not intersection points of the six lines.  They also  are taken outside of $\wti\ell,$ whence
$q_0(T_i)=0,\ i=1,2,3.$

Consider the set of nodes

$$\mathcal C:=\mathcal X \setminus \left(\wti\ell \cup \mathcal A_5\right),\quad |\mathcal C|=28-5-5=18.$$
The following cases of distribution of these $18$ nodes in the lines
$\ell_2,\ldots, \ell_6$ in some order are possible:
\begin{equation*}\label{eq:18cases}
  \begin{matrix}
    {(1)}\  (6,6,6,0,0);\    &{(2)}\   (6,6,5,1,0);\  &{(3)}\   (6,6,4,2,0);\  &{(4)}\  (6,6,4,1,1); \\
      {(5)}\  (6,6,3,3,0);\     &{(6)}\  (6,6,3,2,1);\  &{(7)}\    (6,6,2,2,2);\  &{(8)}\  (6,5,5,2,0);\\
     {(9)}\  (6,5,5,1,1);\      &{(10)}\   (6,5,4,3,0);\ &{(11)}\   (6,5,4,2,1);\  &{(12)}\  (6,5,3,3,1);\\
       {(13)}\  (6,5,3,2,2);\      &{(14)}\   (6,4,4,4,0);\  &{(15)}\    (6,4,4,3,1);\  &{(16)}\  (6,4,4,2,2);\\
       {(17)}\  (6,4,3,3,2);\      &{(18)}\  (6,3,3,3,3);\  &{(19)}\    (5,5,5,3,0);\  &{(20)}\   (5,5,5,2,1);\\
          {(21)}\  (5,5,4,4,0);\      &{(22)}\  (5,5,4,3,1);\  &{(23)}\   (5,5,4,2,2);\  &{(24)}\  (5,5,3,3,2);\\
           {(25)}\  (5,4,4,4,1);\      &{(26)}\   (5,4,4,3,2);\  &{(27)}\   (5,4,3,3,3);\  &{(28)}\  (4,4,4,4,2);
         \end{matrix}
\end{equation*}
\vspace{-.4cm}

\noindent \ $(29)\  (4,4,4,3,3).$
\vspace{.3cm}

We assume  for the convenience that the lines are in the increasing order.

We may assume also that in each above distribution the listed zeros are primary in the respective lines. Indeed, by reordering the lines and making the zeros primary we will get another distribution listed above.

Now one can verify readily that the cases (3)-(29) are not possible, since by adding three arbitrary points $T_i,\ i=1,2,3,$ we make the polynomial $q_0$ to have at least $(6,5,4,3,2)$ primary zeroes
in the lines  $\ell_2,\ldots, \ell_6.$

For example, for several particular cases below, we add the three points to the lines $\ell_2,\ldots, \ell_6,$ according to the following distributions:

\noindent${(3)}\   (0,0,0,1,2); \quad \ {(14)}\    (0,1,0,0,2);\quad \ {(25)}\  (1,1,0,0,1);\quad{(29)}\  (2,1,0,0,0).$

This implies that
$q_0= \ell_2\cdots \ell_6$ hence $p_0=\wti\ell \ell_2\cdots \ell_6= p^\star_{A_1}.$ 
Therefore we get $p^\star_{A_1}\in \mathcal P_4,$ which is a contradiction.

Then note that also the case (1) is not possible since the curve $\wti\ell\ell_2\ell_3\ell_4\in\Pi_4$ contains $23=5+6+6+6$ nodes, while a maximal quartic contains $22=7+6+5+4$ nodes.
Thus the only possible case is the distribution (2). 

Evidently, the curve $\mu_4:=\wti\ell\ell_2\cdots\ell_4$ here is a maximal curve. Hence the node in the line $\ell_5$ together with the five nodes of $\mathcal A_5,$ use the lines  $\ell_2,\ldots, \ell_4.$  Thus the six nodes besides $\wti\ell,$ share also the
three lines $\ell_2,\ell_3,\ell_4,$ passing through $6,6,$ and $5$ primary nodes.

Thus the distribution (2) may correspond only to the following
m-d sequences: $(6,6,6,\wti 4,3,2)$ and $(6,6,\wti 5,5,3,2).$
\end{proof}

\subsubsection*{A $6$-node line shared}

Finally suppose that $\wti\ell$ is a $6$-node line. For the $\wti\ell$-m-d sequence of a node $A\notin \wti\ell$ there are only the following possibilities:
\begin{equation*}\label{eq:9'cases}
  \begin{aligned}
  \text{(i)}   & \quad (\wti 6,6,6,4,3,2)
   &\text{(ii)}  & \quad (\wti 6,6,5,5,3,2)
   &\text{(iii)} & \quad (\wti 6,6,5,4,4,2)\\
  \text{(iv)}  & \quad (\wti 6,6,5,4,3,3)
    &\text{(v)}   & \quad (\wti 6,6,4,4,4,3)
     &\text{(vi)}  & \quad (\wti 6,5,5,5,4,2)\\
   \text{(vii)}  & \quad (\wti 6,5,5,5,3,3)
      &\text{(viii)} & \quad (\wti 6,5,5,4,4,3)
      &\text{(ix)}  & \quad (\wti 6,5,4,4,4,4).
  \end{aligned}
\end{equation*}

\begin{proposition}\label{prp:6810}
Assume that $\Xset$ is a $GC_6$ set without a maximal line, and
$\wti\ell$ is a $6$-node line. Assume also that $\wti\ell$ is used by eight nodes of $\Xset.$ Then it  is used by exactly ten nodes of $\Xset.$

Moreover, these ten nodes form a
$GC_3$ set and share two more lines with six primary nodes each.
Furthermore,
each of these ten nodes has the m-d sequence $(6,6,6,4,3,2).$
\end{proposition}

\begin{proof}
Since $\wti\ell$ is used by at least eight
nodes, we have that $\#\Nset_{\wti\ell} \leq 28-(6+8) = 14.$ By Proposition \ref{prp:kell} the set
$\Nset_{\wti\ell}$ is $5$-dependent. Since $14=3\times 5-1$, one
may apply Corollary~\ref{cor:3n-1} to conclude that either $\Nset_{\wti\ell}$
contains $5+2=7$ collinear nodes, which contradicts the hypothesis, or $12\ (=2\cdot5+2)$ nodes there are in a
conic $\beta.$ Thus the latter case takes place and
$\#\Nset_{\wti\ell} \geq 12.$

Now note that $\Nset_{\wti\ell} \subset\beta.$ Indeed, we may have one or two nodes in $\Nset_{\wti\ell}$ outside of $\beta.$
But in this case  those nodes evidently have fundamental polynomial of degree $3,$  for the set $\Nset_{\wti\ell},$ contradicting Proposition \ref{prp:kell}, (i).

Then let us show that   $\#\Nset_{\wti\ell} = 12.$
Assume by way of contradiction that there are $\ge 13$ nodes in $\Nset_{\wti\ell}.$ Then there are at most $9$ nodes outside of $\beta\cup\wti\ell$ and therefore they are contained in a cubic
$\gamma.$ Then we readily get that $\Xset\subset \wti\ell \beta\gamma\in\Pi_6,$ which contradicts Proposition \ref{correctii}.

Finally note that $\wti\ell \beta$ contains $18$ nodes, i.e., is a maximal cubic.
Therefore, by Proposition \ref{maxcor}, it is used by all the $10$ nodes in
$\Xset\setminus(\wti\ell\cup\beta)$, and hence $\beta$ has to be the product
of two $6$-node lines.
\end{proof}

\begin{proposition}\label{prp:66710}
Assume that $\Xset$ is a $GC_6$ set without a maximal line, and
$\wti\ell_i,\ i=1,2,$ are two disjoint $6$-node lines. Assume also that six nodes of $\Xset$are using $\wti\ell_1$ and $\wti\ell_2.$ Then, the six nodes besides $\wti\ell_1$ and $\wti\ell_2$
share either one more line with $6$ primary nodes or two more lines each with $5$ primary nodes. In the first case the lines $\wti\ell_1$ and $\wti\ell_2$  are used by exactly ten nodes of $\Xset$  and in the second case they are used by exactly six nodes of $\Xset.$

Moreover, in the first and second cases the ten and six nodes form a
$GC_3$ and $GC_2$ sets, respectively.
Furthermore,
each of the ten nodes and each of the six nodes has the m-d sequence $(6,6,6,4,3,2),$ and $(6,6,5,5,3,2),$
respectively.\end{proposition}

\begin{proof}
We have that $\#\Nset_{\wti\ell_1, \wti\ell_2} \leq 28-(6+6+6) = 10.$ By Proposition \ref{prp:kell}, the set
$\Nset_{\wti\ell_1, \wti\ell_2}$ is $4$-dependent. Since $10=3\times 4-2=2\times 4+2$, we
can apply Corollary~\ref{cor:3n-1} and conclude that either $\Nset_{\wti\ell_1, \wti\ell_2}$
contains $4+2=6$ nodes lying in a line $\wti\ell_3$,
or all the ten nodes are lying in a conic $\beta.$

In the first case we readily conclude that $\wti\ell_1\wti\ell_2\wti\ell_3$ is a maximal cubic with $18$ nodes and hence the remaining ten nodes of $\Xset$ are using it.

In the second case we readily conclude that $\beta \wti\ell_1\wti\ell_2$ is a maximal quartic  with $22$ nodes and hence the remaining six nodes of $\Xset$ are using it. Hence the conic $\beta$ reduces to two lines with $5$ primary nodes.

It remains to mention that if a seventh node uses the lines  $\wti\ell_1$ and $\wti\ell_2$
then we get $\#\Nset_{\wti\ell_1, \wti\ell_2} \leq 28-(6+6+7) = 9=2\times 4+1$ which readily reduces to the first case.
\end{proof}

The following table is an analog of one in \cite{HJZ14}. It is obtained from Propositions~\ref{prp:2nl} - \ref{prp:6810}, and shows how many times at most a line
$\wti\ell,$ under certain restrictions, can be used, provided that the $GC_6$-set has no maximal line.

\begin{equation}\label{eq:table}
  \begin{array}{c|c|c|c|}
    & \multicolumn{3}{c|}{\text{maximal \#\ of nodes
      using $\wti\ell$}} \\\hhline{~---}
    \text{total \#} & \text{in general} & \text{no node uses} &
      \text{no node uses} \\
    \text{of nodes}  & & (6,6,6,4,3,2) & (6,6,6,4,3,2), \\
    \text{on $\wti\ell$} & \text{\phantom{no node uses}} &
      \text{constellation} & (6,6,5,5,3,2) \\\hline
        6 & 10 & 7 & 7 \\\hline
    5 & 6 & 6 & 4 \\\hline
    4 & 6 & 3 & 3 \\\hline
    3 & 3 & 3 & 1 \\\hline
    2 & 1 & 1 & 1 \\\hline
  \end{array}
\end{equation}

\vspace{.5cm}

\subsection{The main result}\label{case666432}

In this paper we will prove the following
\begin{proposition}\label{thm:main3} Assume that $\Xset$ is a $GC_6$ set with no maximal line. Then for no node in $\Xset$ the m-d sequence is $(6,6,6,4,3,2)$.
\end{proposition}
Assume by way of contradiction that for a node
in $\Xset$  the m-d sequence is  $(6,6,6,4,3,2)$. Let $(\alpha_1,\ldots,\alpha_6)$ be a respective m-line sequence.

Set  $\Xset=\Aset \cup \Bset$ (see Fig. \ref{case1}) with
\begin{equation*}\label{6}
  \Aset = \Xset \cap \{\alpha_1 \cup \alpha_2 \cup \alpha_3 \}, \quad
  \#\Aset = 18, \quad \text{and} \quad
  \Bset = \Xset \setminus \Aset, \quad  \#\Bset = 10.
\end{equation*}

Denote $\Lset_3:=\{\alpha_1,\alpha_2,\alpha_3\}.$
Note that no intersection point of the three lines of $\Lset_3$ belongs to $\Xset$.
The following is the analogue of  \cite{HJZ14}, Lemma 3.2.
\begin{lemma}\label{lem:Bset}
\nix\vspace{-2mm}
\begin{enumerate}
\setlength{\itemsep}{0mm}
\item
The set $\Bset$ is a $GC_3$ set, and each node $B\in\Bset$ uses the three lines of $\Lset_3$ and the three lines it uses within $\Bset$, i.e.,
\begin{equation}\label{ban}
  p_{B,\Xset}^\star = \alpha_1\alpha_2\alpha_3 p_{B,\Bset}^\star\,.
\end{equation}
\item
No node in $\Aset$ uses any of the lines of $\Lset_3.$
\end{enumerate}
\end{lemma}
\vspace{-.2cm}

\begin{figure}
\begin{center}
\includegraphics[width=10.0cm,height=5.cm]{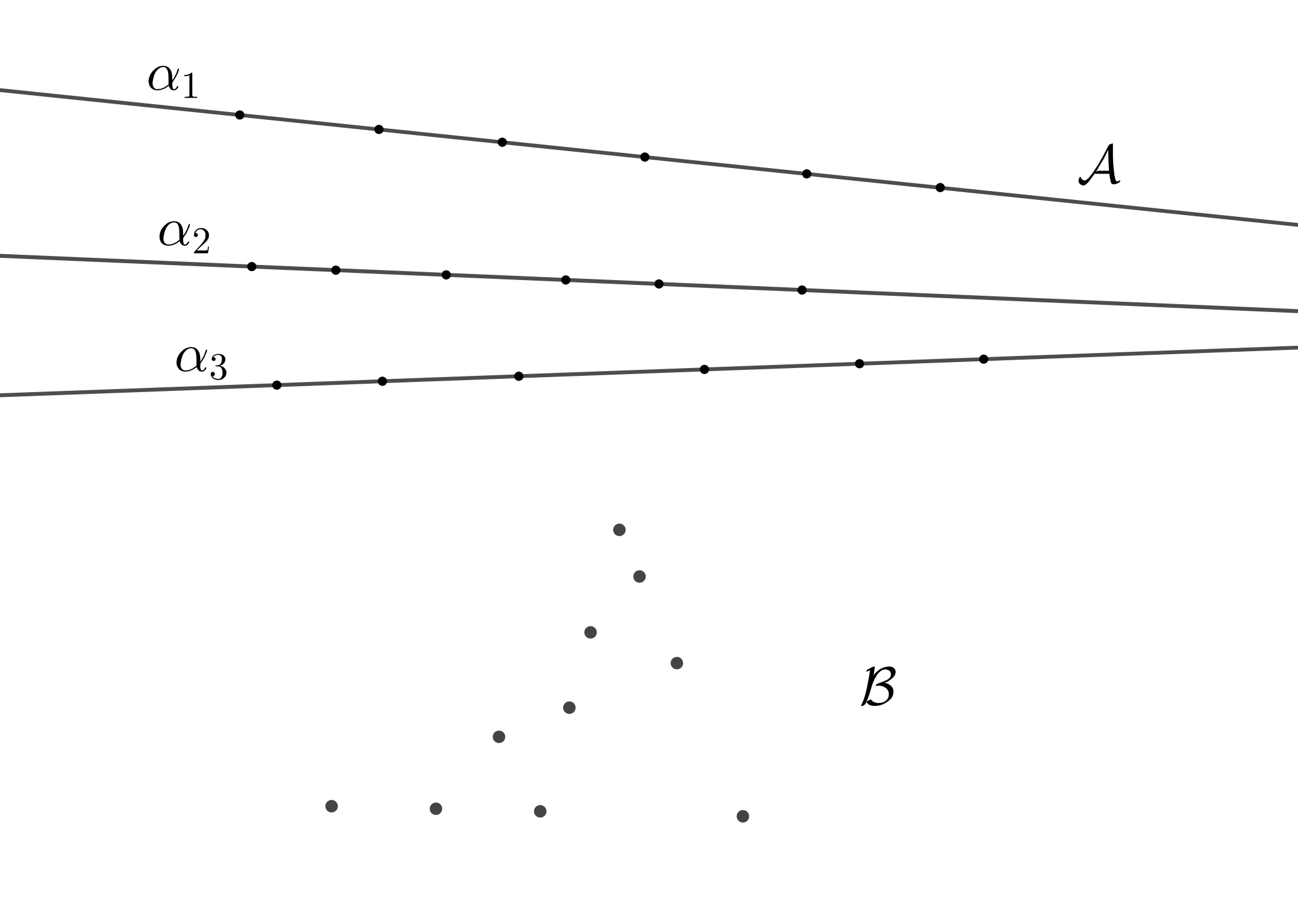}
\end{center}
\caption{The case
$(6,6,6,4,3,2)$ with $\Xset=\Aset\cup\Bset.$} \label{case1}
\end{figure}

\begin{proof}
(i) Suppose by way of contradiction that the set $\Bset$ is not $3$-correct, i.e., it is a subset of a cubic $\gamma_0.$ Then $\Xset$ is a subset of the zero set of the polynomial $\alpha_1\alpha_2\alpha_3 \gamma_0\in \Pi_6,$ which contradicts Proposition \ref{correctii}.

Now, we readily obtain the formula \eqref{ban}.

(ii) Without loss of generality assume that $A \in \alpha_1$ uses the line $\alpha_2.$ Then $p_{A}^\star = \alpha_2\,q,$ where $q \in \Pi_5.$ It is easily seen that $q$ has (6,5) primary zeros in the lines $(\alpha_3 , \alpha_1).$ Therefore, in view of Corollary \ref{cor22}, we obtain that $p_{A}^\star =\alpha_1\alpha_2\alpha_3r,\ r \in \Pi_3,$ which is a contradiction.
\end{proof}

\begin{lemma} \label{lem:26664}
No node from $\Aset$ can have the m-d sequence $(6,6,6,4,3,2).$
\end{lemma}
\begin{proof}
Assume conversely that $A\in\Aset$ has the m-d sequence $(6,6,6,4,3,2).$ Denote a respective m-line sequence by $(\alpha_1',\ldots,\alpha_6').$ The lines here, according to Lemma \ref{lem:Bset}, (ii), are different from $\alpha_1,\alpha_2,\alpha_3.$

Denote $\Aset' = \Xset \cap \{\alpha_1' \cup \alpha_2' \cup \alpha_3' \}.$
The three lines
 $\alpha_1',\alpha_2',\alpha_3'$ contain at least $9=3+3+3$ nodes outside of $\gamma:=\alpha_1 \cup \alpha_2 \cup \alpha_3.$ The fourth line $\ell_4'$ contains at least $1=4-3$ node outside of $\gamma$ denoted by $C.$ Since $\#\Bset=10$ we conclude that these four  lines have exactly $10$ nodes in $\Bset$ and $12=4\times 3$ nodes in $\Aset.$ Therefore we obtain that $\Bset\subset\alpha_1' \cup \cdots \cup \alpha_4',$ and
 $C\in\Xset\setminus (\Aset\cup\Aset').$

 This, in view of Lemma \ref{lem:Bset}, (i), implies that
 $p_C^\star = \alpha_1\alpha_2\alpha_3\alpha_1'\alpha_2'\alpha_3'.$

 From here we readily conclude that the node $C$ uses six lines none of which is a maximal line within $\Bset.$ Indeed, we have that $\alpha_i\cap \Bset=\emptyset,\ i=1,2,3,$ and   $|\alpha_i'\cap \Bset|=3,\ i=1,2,3.$ This contradicts  Lemma \ref{lem:Bset}, (i).
\end{proof}

\begin{definition} We say, that a line $\ell$ is a $k_\Aset$-\emph{node} line if it passes through exactly $k$ nodes of $\Aset,\ k=0,1,2,3.$\end{definition}

\begin{lemma}\label{lem01}
(i) Assume that a line $\wti\ell\notin\Lset_3$  does not intersect  a line $\alpha\in\Lset_3$ at a node in $\Xset.$ Then the line
$\wti\ell$ can be used by atmost $1$ node from $\Aset.$
Moreover, this latter node  can belong only to $\alpha.$ \par
(ii) If $\ell$ is $0_{\Aset}$ or $1_{\Aset}$-node line then no node from $\Aset$ uses it.
\par
(iii) If $\ell$ is $2_{\Aset}$-node line then it can be used  by atmost one node from $\Aset.$ 
\end{lemma}
\begin{proof}
(i) Without loss of generality assume that $\alpha=\alpha_1$ and $A \in \alpha_2 $ uses $\wti\ell:\
    p_{A}^\star =\wti\ell\,q,\ q \in \Pi_4.$
It is easily seen that $q$ has $(6,5,4)$ primary zeros in the lines $( \alpha_1 ,  \alpha_3 ,  \alpha_2).$ Therefore, in view of Corollary \ref{cor22}, we conclude that
$p_{A}^\star = \wti\ell\, \alpha_1\, \alpha_2\, \alpha_3\,r,\ r\in \Pi_1,$
which is a contradiction.

Now assume conversely that $A,B \in  \alpha_1\cap\Xset$ use the line $\wti\ell.$ Choose a point $C \in  \alpha_2\setminus(\wti\ell\cup\Xset).$ Then, in view of Lemma \ref{lem:mm-1}, choose numbers $a$ and $b,$ with $|a| + |b| \neq 0,$ such that $p(C)=0,$ where $p:=a p_A^\star + b p_B^\star.$ It is easily seen that $p =\wti\ell\,q,\ q \in \Pi_5$ and  the polynomial $q$ has $(6,5,4)$ primary zeros in the lines $( \alpha_2, \alpha_3,\alpha_1).$ Therefore $p = \wti\ell \alpha_1\, \alpha_2\, \alpha_3\,q,$ where $q\in\Pi_2.$ Thus $p(A) = p(B) = 0,$ implying that $a=b = 0$, which is a contradiction.

The items (ii) and (iii) readily follow from (i).
\end{proof}
A node is called an $i_m$-\emph{node}, $i\le 2,$ if it lies in exactly $i$ maximal lines.
\begin{lemma}\label{lem02}
Let $\wti\ell$ be a $3_{\mathcal A}$ type line passing through a $2_m$-node $B\in\Bset.$ Assume also that  the node set $\mathcal B\setminus \{\wti\ell\}$ contains $4$ collinear nodes.  Then the line $\wti\ell$ can be used by at most three nodes from $\Aset.$
\end{lemma}

\begin{proof} Assume by way of contradiction that the line $\wti\ell$ is used by four nodes from a set $\mathcal A_4:=\{A_1,\ldots, A_4\}\subset\Aset.$
For any chosen points $T_i,\ i=1,2,3,$ (see the proof of Proposition \ref{prp:5nl} for the details) there is a polynomial
\begin{equation}\label{p0}p_0\in \mathcal X_4:= linear span \{p^\star_{A_1},\ldots,p^\star_{A_4}\},\quad p_0\neq 0,\end{equation}

\noindent such that $p_0(T_i)=0,\ i=1,2,3.$
On the other hand we have that
$$p_0=\wti\ell q_0,\quad q_0\in\Pi_5,\ \hbox{and}\ q_0(T_i)=0,\ i=1,2,3.$$
Now consider the set of nodes
\vspace{.3cm}

$\qquad\qquad\qquad\mathcal C:=\mathcal A \setminus \left(\wti\ell \cup \mathcal A_4\right),\quad |\mathcal C|=18-3-4=11.$

Denote by $\ell^*$ the line passing through the four collinear nodes of $\mathcal B\setminus \{\wti\ell\}.$

The following cases of distribution of above $11$ nodes in the three lines of
$\Lset$ are possible:

${(1)}\ (5,5,1); \quad {(2)}\ (5,4,2); \quad {(3)}\ (5,3,3); \quad {(4)}\ (4,4,3).$

One can verify that the cases (1)-(4) are not possible in the following way. By locating conveniently the three points $T_i,\ i=1,2,3,$ we make the polynomial $q_0$ to have at least $(6,5,4,3)$ primary zeroes
in the lines  $\alpha_1,\alpha_2, \alpha_3, \ell^*,$ in some order.
Thus we get  $q_0=\alpha_1\alpha_2 \alpha_3\beta,\ \beta\in\Pi_2,$
implying $p_0=\alpha_1\alpha_2 \alpha_3\gamma,\ \gamma\in\Pi_3.$
Hence, in view of Lemma \ref{lem:Bset}, (i), we readily get that $p_0\in linear span \{p^\star_{B,\Xset} : B\in \Bset\},$ which contradicts \eqref{p0}.

To implement the above described verification in details suppose that
$$|\ell^*\cap \mathcal C|=k,\  k\le 2.$$

  {\bf Case 1: $k=0.$}
In the case of distribution (1), $(5,5,1),$ we add the three points in the form $(1,0,2),$ meaning that we add a point to the line $\alpha_1$ and the remaining two points to the line $\alpha_3.$
In the case of distributions (2)-(4) we add the three points in the form  $(1,1,1),\ (1,2,0),$ $(2,1,0),$ respectively.

Then note that the polynomial $q_0$ has at least $(6,5,4,3)$ primary zeroes
in the lines  $\alpha_1,\alpha_2, \ell^*, \alpha_3,$ in the indicated order.

 {\bf Case 2: $k=1.$}
 In this case a node denoted by $A^*$ in $\Cset$  belongs to the line $\ell^*.$
 The following are the cases of distribution of remaining $10$ nodes of $\Cset$ in the lines of $\Lset$:

${(1')}\ (5,5,0); \quad {(2')}\ (5,4,1); \quad {(3')}\ (5,3,2); \quad {(4')}\ (4,4,2); \quad {(5')}\ (4,3,3).$

  Consider the distribution sequence $(1'),\ (5,5,0).$ In this case we have that $\Aset_4 \cup \{A^*\}\subset \alpha_3.$
Note that this is the only case when instead of $\ell^*$ we use the two maximal lines passing through $B\in\Bset,$ denoted by  $\ell^{**}_1$ and $\ell^{**}_2$. Each of these lines passes through $3$ nodes in $\Bset\setminus\wti\ell.$
Note that these lines do not pass through $A^*$ since they intersect $\ell^*$ at $B\in\Bset.$ 

Thus in case  ${(1')}$ we add a point to the line $\ell^{**}_1.$ 
Then we add the remaining two points to the lines of $\Lset$ in the form  $(1,0,1).$  Now note that the polynomial $q_0$ has at least $(6,5,4,3,2)$ zeroes in the following ordered lines: $\alpha_1,  \alpha_2, \ell^{**}_1,\ell^{**}_2,\alpha_3,$ counting also $A^*\in\alpha_3.$

In the remaining cases $(2')-(5')$ we add a point to the line $\ell^*$ to have there $6$ zeroes and use it as the first line in the ordered line sequence.
Then we add the remaining two points to the lines of $\Lset$ in the form  $(0,0,2), (0,1,1),\ (1,0,1)\\ (1,1,0),$ respectively.

Finally notice that the polynomial $q_0$ has at least $(6,5,4,3)$ zeroes in the ordered lines: $\ell^*, \alpha_1,  \alpha_2, \alpha_3.$

 {\bf Case 3: $k=2.$}
 In this case two nodes of $\Aset$ belong to the line $\ell^*.$
 The following are the cases of distribution of the remaining $9$ nodes of $\Cset$ in the three lines of
$\Lset$:

\noindent ${(1'')} (5,4,0); \ {(2'')} (5,3,1);\ {(3'')} (5,2,2); \ {(4'')} (4,4,1);\ {(5'')} (4,3,2);\ {(6'')} (3,3,3).$

In these cases the line $\ell^*$ has $6$ zeroes and is the first line in the ordered sequence of lines.
For the distributions $(1'')-(6'')$ we add the three points in the form  $(0,0,3), (0,1,2),\ (0,2,1)\ (1,0,2),\ (2,1,0),$ respectively.

Then note that as above the polynomial $q_0$ has at least $(6,5,4,3)$ zeroes in the following ordered lines: $\ell^*, \alpha_1,  \alpha_2, \alpha_3.$
\end{proof}

 \subsection{The proof of the main result}

Consider all the lines passing through a node $B\in\Bset$ and at least one more node of ${\mathcal X}$. Denote the set of these lines by $\Lset(B).$  Let $m_k:=m_k(B),\ k=1,2,3,$ be the number of $k_{\Aset}$-node lines from $\Lset(B).$
Then the following holds:
\begin{equation}\label{eq18}
  1m_1(B) + 2m_2(B) + 3m_3(B)   =
  \# \Aset = 18 .
\end{equation}
 \begin{lemma}\label{lem:m36} We have that
 $m_3(B) \leq 5.$
 \end{lemma}
 \begin{proof}
 The relation \eqref{eq18} implies that $m_3(B) \leq 6.$ Assume by way of contradiction that six lines pass through $B$ and three nodes in $\Aset.$
Therefore these six lines intersect the three lines $\alpha_1, \alpha_2, \alpha_3,$ at all the $18$ nodes of $\Aset.$

\noindent Note that $\alpha_1\alpha_2\alpha_3$ is a maximal cubic. Hence, by Proposition \ref{maxcor}, the six lines contain as components the lines $\alpha_1, \alpha_2, \alpha_3,$ which is a contradiction.
 \end{proof}
\emph{The proof of Proposition \ref{thm:main3}.}
  We will prove Proposition in   three steps. Recall that the set $\Bset$ is a $GC_3$ set.

\emph{\bf Step 1.} The set $\Bset$ is a Chung-Yao set (with $5$ maximal lines, Fig. \ref{case1}).

Let us fix as a  node $B\in\Bset.$ Note that all nodes in this case are $2_m$-nodes.
According to Lemma \ref{lem02} any $3_\Aset$ type line $\wti\ell$ here is used by at most $3$ nodes of $\Aset.$ Indeed,  $\wti\ell$ passes through at most two nodes of $\Bset.$ Thus it intersects at most $4=2\times 2$ maximal lines of $\Bset$ and the four nodes of the fifth maximal line of $\Bset$  are outside of $\Bset\setminus \wti\ell$ (see Fig. \ref{case1}).

Therefore, in view of Lemma \ref{lem01},  the number of usages of the lines $\wti\ell$ through $B$ with the nodes from $\Aset$ equals at most:
$$m_2(B)+3m_3(B) \ge 18 =m_1(B) +2m_2(B)  + 3m_3(B).$$

Hence
$m_1=m_2=0$ and $m_3=6,$ which contradicts Lemma \ref{lem:m36}.

\emph{\bf Step 2.} $\Bset$ is a Carnicer-Gasca set (with $4$ maximal lines, Fig. \ref{case2}).

We have that there are at most four $3$-node lines in $\Bset$ (\cite{GV1}, Prop. 5). Moreover, any $3$-node line passes through $1_m,\ 1_m,\ 1_m,$ or through $2_m,\ 1_m,\ 1_m,$ nodes \cite{GV1}.
There are exactly six  $2_m$-nodes in $\Bset.$ Therefore we have at least two  nodes in $\Bset$, denoted by $B_0$ and $B_1,$ through which no $3$-node line passes.

Denote the line passing through the nodes $B_0$ and $B_1,$ by $\ell_{01}.$
\begin{figure}
\begin{center}
\includegraphics[width=10.0cm,height=5.cm]{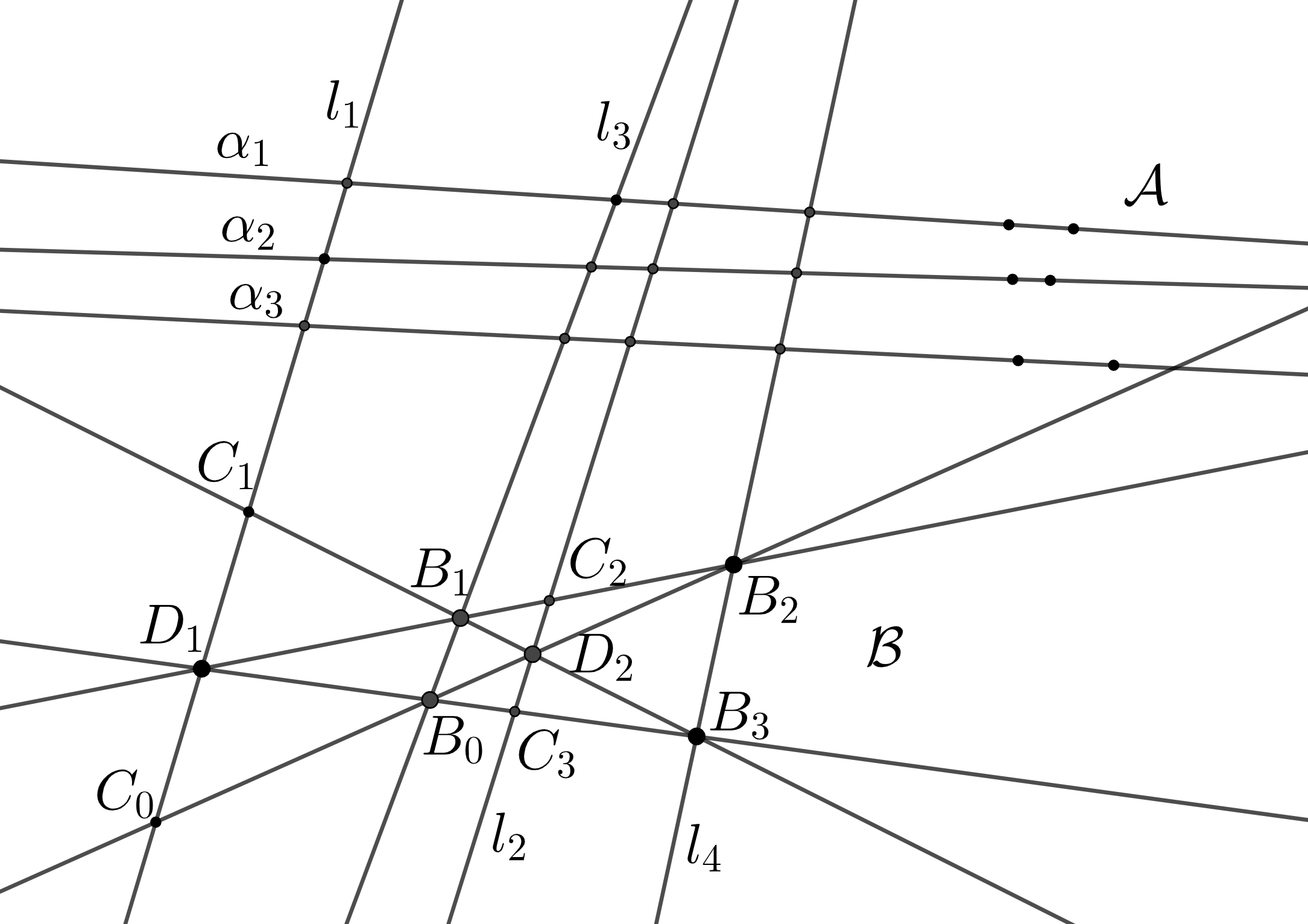}
\end{center}
\caption{The set $\Bset$ is a Carnicer-Gasca set} \label{case2}
\end{figure}
\begin{lemma}\label{lem:B0} The line $\ell_{01}$  is a $3_\Aset$ type $5$-node line and is used by exactly six nodes from $\Aset.$
\end{lemma}
\begin{proof} Assume by way of contradiction that $\ell_{01}$ is a $\le 4$-node line. Then, in view of Lemma \ref{lem01}, it is used $\le 1$ times from $\Aset.$ On the other hand if $\ell_{01}$ is a $5$-node line but is not used six times from $\Aset$ then, according to Propostion
\ref{prp:5nl}, it  is used $\le 4$ times from $\Aset.$

Note that there is no $6$-node line through $B:=B_0,$ since there is no $3$-node line through it in $\Bset.$
Next, by Lemma \ref{lem02},  any $5$-node line through $B,$ except of $\ell_{01},$ is used by $\le 3$ nodes from $\Aset$ (see Fig. \ref{case2}).

Thus, in view of Lemma \ref{lem01},  we have that the number of usages of the lines through $B$ with the nodes from $\Aset$ equals at most:
$$m_2(B)+4+3(m_3(B)-1) \ge 18 =m_1(B) +2m_2(B)  + 3m_3(B).$$
Hence
$m_1+m_2\le 1,$ and $m_3\ge 6,$
which contradicts Lemma \ref{lem:m36}.
\end{proof}
Denote by $\Aset_6\subset\Aset$ the set of six nodes that are using the line $\ell_{01}.$

We get from Proposition \ref{prp:5nl} that   the six nodes of $\Aset_6$ besides $\ell_{01},$ share also
three other lines  passing through $6,6,5$ primary nodes, respectively. Furthermore the m-d sequence for the six nodes is $(6,6,6,4,3,2),$ or $(6,6,5,5,3,2).$

Lemma \ref{lem:26664} implies that the first case cannot take place.
Thus the m-d sequence for all nodes of $\Aset_6$ is  $(6,6,5,5,3,2).$
Note that, in view of \eqref{maxmax}, $\Aset_6$ is a $GC_3$ set, since $\ell_1\cdots\ell_4$ is a maximal quartic.

Let $\ell_1, \ldots, \ell_6$ be a respective m-line sequence, where $\ell_3:=\ell_{01}.$ Note that $\ell_1, \ldots, \ell_4$ are invariable lines, and $\ell_5,\ell_6$ are variable lines, for the nodes of $\Aset_6.$ The lines $\ell_1, \ldots, \ell_4$ have at least $10=3+3+2+2$ distinct nodes in $\Bset.$ Since $\#\Bset=10$ we conclude that these four  lines have exactly $10$ nodes in $\Bset$ and $12=4\times 3$ nodes in $\Aset$ (see Fig. \ref{case2}).

 Recall that we have four maximal lines in $\Bset$
which are $4$-node lines and intersect in primary nodes of $\Xset$ each of the lines $\ell_1, \ldots, \ell_4.$ The nodes $B_0$ and $B_1$ are in the line $\ell_3.$ Denote by $B_2,B_3$ the two primary nodes in  $\ell_4\cap\Bset.$

We readily conclude that $B_i,\ i=0,\ldots,3,$ are $2_m$-nodes of $\Bset$ and the $4$ maximal lines are the lines $\ell_{02}, \ell_{03}, \ell_{12}, \ell_{13},$ where $\ell_{ij}$ is the line passing through the nodes $B_i$ and $B_j$ (see Fig. \ref{case2}).

 The remaining two $2_m$-nodes of
$\Bset$ are the remaining two intersection points of the maximal lines denoted by $D_1:=\ell_{03}\cap\ell_{12}$ and $D_2:=\ell_{02}\cap\ell_{13}.$

\noindent The nodes $D_1$ and $D_2$ one by one lie in the lines $\ell_1$ and $\ell_2,$  respectively, since the latters  are $3$-node lines within $\Bset$ and each contains at most  one $2_m$-node.

\begin{lemma}\label{lem:B2B3}  The following is true for at least one of $B\in\{B_2,B_3\}:$

\noindent No $3$-node line within $\Bset$ passes through the node $B.$
\end{lemma}
\begin{proof} Consider the node $B_2.$ We have that
$\Xset\setminus (\ell_{02}\cup\ell_{12})=\{B_3,C_1,C_3\}.$
Since $\ell_{23}$ is a $2$-node line in $\Bset$ we get that the only candidate for $3$-node line through $B_2$ is the line passing through the nodes $B_2,C_1,C_3,$ provided that the latter triple of nodes is collinear (see Fig. \ref{case2}).

Similarly we get that the only candidate for a $3$-node line through $B_3$ is the line passing through the nodes $B_3,C_0,C_2,$ provided they are collinear.

What we need to show is that at least one of the two triples of nodes is not collinear.
Assume by way of contradiction that the both triples are collinear, lying in some two lines $\ell_0$ and $\ell_0',$ respectively.

Consider the following collinear sets
$$\Xset_1:=\{C_0,D_1,C_1\},\ \Xset_2:=\{C_2,D_2,C_3\},\ \Xset_3:=\{B_2,B_3\}.$$
It is easily seeen that the four maximal lines of $\Bset$ pass through one point from each of
$\Xset_1,\Xset_2, \Xset_3.$ Note that the above two lines $\ell_0$ and $\ell_0',$
have the same property.
Therefore we get $\#\Lset_{\Xset_1,\Xset_2, \Xset_3}\ge 6,$ which contradicts \eqref{M332}.
\end{proof}
Thus from now on one can assume, without loss of generality, that no $3$-node line passes through the node $B_2.$

\begin{lemma} \label{33node} The set $\Bset,$ except of the lines $\ell_1$ and $\ell_2,$ may have just one more $3$-node line, which passes through the nodes $B_3, C_0, C_2,$ provided that the latter nodes are collinear.

Moreover, $(\ell_1, \ell_2)$ is the only disjoint pair of $3$-node lines in $\Bset.$
\end{lemma}

\begin{proof} Indeed, any $3$-node line must pass through a node of $\Bset$ outside of the lines $\ell_1$ and $\ell_2,$ i.e., through one of $B_0,\ldots,B_3.$ Thus it must pass through $B_3$  and therefore also through $C_0$ and $C_2,$ provided that the latter triple of nodes is collinear (see Fig. \ref{case2}). It remains to note that the third $3$-node line intersects both of the lines $\ell_1$ and $\ell_2$ at nodes of $\Bset.$\end{proof}

\begin{lemma}\label{lem:B012} There is a type $3_\Aset$ $4$-node line through each of the nodes  $B_0, B_1,$ and $B_2.$ Moreover, these lines are used by exactly $3$ nodes from $\Aset.$
\end{lemma}
\begin{proof} Denote by $B$ any of the nodes $B_0, B_1,B_2.$
Note that there are no $6$-node lines through $B.$
Then we get, from Lemma \ref{lem02},  that any type $3_\Aset$   $5$-node line through $B,$ except of $\ell_3$ or $\ell_4$ is used by $\le 3$ nodes from $\Aset.$

Thus,  the number of usages of the lines through $B$ with the nodes from $\Aset$ equals at most:
\begin{equation}\label{6}m_2(B)+6+3(m_3(B)-1) \ge 18 =m_1(B) +2m_2(B)  + 3m_3(B).
\end{equation}
Hence we obtain that
$m_1+m_2\le 3.$

Let us denote by $max_{use}:=max_{use}(B)$ the maximum possible usage of lines through $B$ we obtained, i.e., $max_{use}=m_2+3m_3+3.$

Then, as it follows from the equality in \eqref{6}, the quantity $m_1 +2m_2$ is divisible by $3.$
Thus
the following four cases are possible here.

1) $m_1=m_2=0,$ then $m_3=6,$ which contradicts Lemma \ref{lem:m36}.

2) $m_1=m_2=1,$ then $m_3=5,\ max_{use}=19,$

3) $m_1=3, m_2=0,$ then $m_3=5,\ max_{use}=18,$

4) $m_1=0, m_2=3,$ then $m_3=4,\ max_{use}=18.$

Note that in each of the above cases 2), 3), 4) there are at least 4,5,4 lines of type $3_\Aset$, respectively. Since in the case 2)  $max_{use}=19=18+1$ one of the type $3_\Aset$ lines may be used
less than $3$  times, or more precisely $2$ times.

It remains to take into account that from these four lines only three may be $5$-node lines. For example, for $B=B_0$ these three lines are $\ell_{B_0B_1}, \ell_{B_0C_1}$ and $\ell_{B_0C_2}$ (see Fig. \eqref{case2}). The remaining one certainly is a $4$-node line.
\end{proof}

Consider a $4$-node line $\wti\ell$ through $B_0, B_1, B_2,$ mentioned in Lemma \ref{lem:B012}, used by exactly $3$ nodes of $\Aset.$
According to Proposition \ref{prp:4nl3} the $\wti\ell$-m-d sequence of each mentioned triple of nodes is either

(a) $(\wti 4,6,6,5,4,2)$ or

(b) $(\wti 4,6,5,5,5,2).$

Note that the first five lines in the respective $\wti\ell$-m-line sequences are invariable for the triples of nodes.

Now assume that the case (a) holds for a triple of nodes. Denote a respective $\wti\ell$-m-line sequence  for the nodes of the triple by $\wti\ell, \ell_2', \ldots, \ell_6'.$
Consider the  m-d sequence for the nodes of $\Aset_6:$  $(6,6,5,5,3,2),$ and the  respective m-line sequence $\ell_1, \ldots, \ell_6.$ Note that $\Aset_6$ is a $GC_3$ set.

We get from Lemma \ref{33node} that the pair of the lines $\ell_2',\ell_3'$ coincides with $\ell_1,\ell_2.$ Then note that in the set $\Bset\setminus (\ell_1\cup\ell_2)$ the only $3_\Aset$ lines with two primary nodes are the lines $\ell_3$ and $\ell_4.$ Thus $\ell_4'$ coincides with one of them.

Now, in view of Corollary \ref{cor22}, we readily obtain that any node $E$ of $\Aset_6,$  besides the lines $\ell_2',\ell_3',\ell_4',$ uses also the lines $\wti\ell$ and $\ell_5'.$ Indeed, in view of Proposition \ref{prp:n+1ell}, the node $E$ uses one of the lines $\wti\ell,\ell_5'$ to which it does not belong. Then we get that $E$ uses also the other line. This is a contradiction, since outside of the curve $\wti\ell\ell_2'\cdots\ell_5'\in\Pi_5$ there are only $3$ nodes.

It remains to consider the case when all the three $\wti\ell$-m-d sequences equal (b). Denote a respective $\wti\ell$-m-line sequence  by $\wti\ell, \ell_2'', \ldots, \ell_6''.$

\begin{lemma}\label{nor} (i) The above three triples are disjoint in this case.

\noindent (ii) Suppose the line $\ell_2''$ with $6$ primary nodes for a triple  coincides with one of the lines $\ell_1$ or $\ell_2.$ Then the triple is not a subset of the set $\Aset_6.$
\end{lemma}
\begin{proof}
(i) Consider a pair of triples. Note that for them a line among the $4$ invariable lines $\ell_2'', \ldots, \ell_5''$ is different. Indeed, assume conversely that all threse lines coincide with each other. Then as above we readily get that also the $4$-node lines $\wti\ell$ coincide, which is a contradiction. 
Thus the invariable lines for each pair of triples differ at least with two lines. Therefore for any variable line the line sequences are different. Hence the triples are disjoint.

(ii) Assume that the line $\ell_2''$ 
coincides, say,  with $\ell_1.$  Then the three invariable lines  $\ell_3'',\ell_4'',\ell_5''$ cannot coincide with $\ell_2,\ell_3,\ell_4.$ Indeed, otherwise any node in $\Aset,$ together with $\ell_2'',\ldots,\ell_5''$ uses also the $4$-node line $\wti\ell,$ which is a contradiction since outside of $\wti\ell,\ell_2'',\ldots,\ell_5''$ there are only $3$ nodes. 

Thus, by taking into account $\wti\ell,$ we have two invariable lines in the m-line sequence of the triple that are not present in $\{\ell_1,\ldots,\ell_4\}.$
Next let us fix the variable line $\ell_5$ such that it differs from the two mentioned lines. Indeed, we may choose as $\ell_5$ any maximal line of the  $GC_3$ set $\Aset_6.$ Thus the considered triple is disjoint with the three nodes of $\Aset_6\setminus\ell_5$ that use the lines $\ell_1,\ldots,\ell_5.$ It remains to note that the triple does not coincide with the three nodes of $\Aset_6\cap\ell_5,$ since the latter three nodes are collinear.\end{proof}

Now suppose that for all three triples the the line $\ell_2''$ with $6$ primary nodes is the third possible $3$-node line of $\Bset,$ different from $\ell_1$ and $\ell_2$ (Lemma \ref{33node}). Then, in view of Lemma \ref{nor}, (i), this line is used by $9=3+3+3$ nodes, which is a contradiction.

Finally suppose that for one of the triples the line $\ell_2''$ coincides with one of the two disjoint $6$-node lines, say with $\ell_1.$  Then, in view of Lemma \ref{nor}, (ii), the line $\ell_1$ is used by at least $7=6+1$ nodes from $\Aset.$ Observe that $\ell_1$ is used by a node from $\Bset$ too. Thus in all the $6$-node line $\ell_1$ is used by at least $8$ nodes. This, in view of Proposition \ref{prp:6810} and Lemma \ref{lem:26664}, is a contradiction.

\emph{\bf Step 3.} The set $\Bset$ is a principal lattice (with three maximal lines).

Concider a $2_m$-node $B\in\Bset.$ 
\begin{figure}
\begin{center}
\includegraphics[width=10.0cm,height=5.cm]{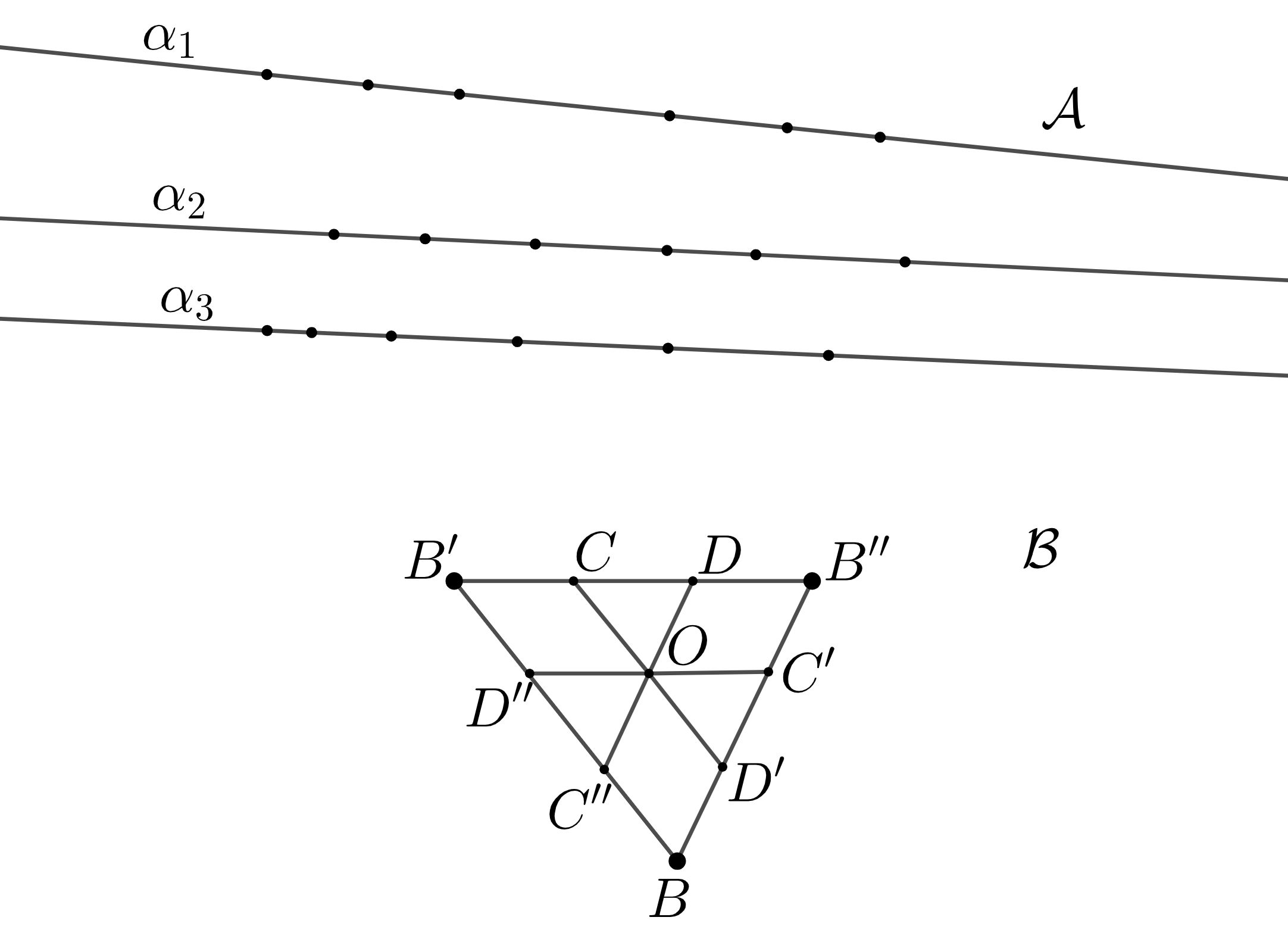}
\end{center}
\caption{ The set $\Bset$ is a principal lattice} \label{case3}
\end{figure}
Note that there is no $3$-node line through $B$ within $\Bset$  (see Fig. \ref{case3}).

Assume that the line $\wti\ell=\ell_{BC}$ (or $\ell_{BD}$) is a $5$-node line and is used $\ge 5$ times from $\Aset.$ Then, according to Proposition \ref{prp:5nl}, it is used by exactly $6$ nodes from $\Aset$ and the $6$ nodes besides $\wti\ell$ share also $3$ lines with $6,6,5$ primary nodes. The two $6$-node lines are $3$-node disjoint lines within $\Bset.$ Thus they pass through the $1_m$ nodes $C,C',C'',$ and $D,D',D'',$ respectively. This is a contradiction since the node $C$ belongs to the line  $\wti\ell$ and is not primary.

Thus the lines $\ell_{BC}$ and $\ell_{BD}$ are used $\le 4$ times from $\Aset.$
Note that the line $\ell_{BO},$ as well as any $4$-node line through $B,$ according to Lemma \ref{lem02}, is used $\le 3$ times from $\Aset.$

Hence the maximal possible number of usages of the lines through $B$ with the nodes from $\Aset$ equals:
\begin{equation*}\label{4+4}m_2(B)+4+4+3(m_3(B)-2) \ge 18 =m_1(B) +2m_2(B)  + 3m_3(B).
\end{equation*}
Hence
$m_1+m_2\le 2.$
Then $m_1 +2m_2$ is divisible by 3 and $max_{use}(B)=m_2+3m_3+2.$
Thus the following two cases are possible:

1) $m_1=m_2=0, m_3=6,$ which contradicts  \ref{lem:m36},

2) $m_1=m_2=1, m_3=5,\ max_{use}=18.$

Now we readily conclude that:

$B1)$ The lines $\ell_{BC}$ and $\ell_{BD}$ are $5$-node lines used exactly $4$ times from $\Aset.$

$B2)$ No node of $\Aset$ may use two lines through $B.$ 

$B3)$ There are at least two $4$-node lines through $B$ that are of type $3_\Aset$ and are used by exactly $3$ nodes from $\Aset.$

$B3')$ Note that if the line $\ell_{BO}$ is not of  type $3_\Aset$ then there are three above mentioned $4$-node lines through $B.$

Then let us consider the usages of lines passing through the $0_m$ node $O\in\Bset$ by the nodes of $\Aset.$
Consider the three lines through $O:\ OC, OC', OC'',$ which can be $6$-node lines.
These lines are used by $3$ nodes of $\Bset$ and therefore, by Proposition \ref{prp:6810}, they can be used by at most $4$ nodes of $\Aset.$

Similarly the lines $OB, OB', OB''$ can be $5$-node lines, in which case, according to Lemma \ref{lem02}, they can be used by at most $3$ nodes from $\Aset.$

Thus for the maximal possible number of usages of the lines through $O$ with the nodes from $\Aset$ we have:
$$m_2(O)+4+4+4+3(m_3(O)-3) \ge 18=m_1(O) +2m_2(O)  + 3m_3(O).$$
Hence
$m_1+m_2\le 3.$
Then $m_1 +2m_2$ is divisible by 3 and $max_{use}(O)=m_2+3m_3+3.$
Thus the following cases are possible

1)  $m_1=m_2=0, m_3=6,$ which contradicts  Lemma \ref{lem:m36}.

2) $m_1=m_2=1, m_3=5,\ max_{use}=19,$

3) $m_1=3, m_2=0, m_3=5,\ max_{use}=18,$

4) $m_1=0, m_2=3, m_3=4,\ max_{use}=18.$

Now we readily conclude that

$O1)$ At least two of the three lines through $O:\ OC, OC', OC'',$  are $6$-node lines and are used exactly $4$ times from $\Aset.$

$O2)$ At most one node of $\Aset$ may use two lines through $O$ all others may use only one line through $O.$

$O3)$ From the six lines through $O:\ OC, OC', OC'', OB,  OB', OB'',$ in view of Lemma \ref{lem:m36}, at most five are of type $3_\Aset$ and possibly are used by $\ge 3$ nodes from $\Aset.$

Thus, by the remarks $B3),$ and $B3'),$ there are at least six, possibly seven,  type $3_\Aset\ 4$-node lines, through the $2_m$-nodes of $\Bset,$ that are used by exactly $3$ nodes from $\Aset.$
In view of Proposition \ref{prp:4nl3} for these $4$-node lines we have one of the following two $\wti\ell$-m-d sequences:

(c) $(\wti 4,6,6,5,4,2),$

(d) $(\wti 4,6,5,5,5,2).$

\begin{lemma}\label{2nd4}
 The second $4$ in the $\wti\ell$-m-d sequence (c), in a respective $\wti\ell$-m line sequence, corresponds to a $3_\Aset$ $4$-node line passing through $B,B'$ or $B''.$
\end{lemma}
\begin{proof}
Suppose that $\ell_1,\ldots,\ell_6$ is a respective $\wti\ell$-m-line sequence.
Then $\ell_2$ and $\ell_3$ pass through the six $1_m$-nodes of $\Bset.$ Thus the line $\ell_4$ coincides with one of the lines $B'O,$ or $B''O,$ say with $B'O$ (see Fig. \ref{case3}).

Now $\ell_5$ passes through $B''$ as the only remaining primary node in $\Bset.$ Let us show that it  does not pass through any other node of $\Bset.$

Note that $\ell_5$ cannot pass through $O$ since then each of the three nodes will use two lines passing through $O,$ which contradicts the remark $O2).$

Then assume conversely that $\ell_5$ passes through one of $1_m$ nodes, say $C'.$
As we know from the remark $B1),$ the line $\ell_{B'C'}$ is used by the fourth node of $\Aset$ denoted by $F.$
In view of Proposition \ref{prp:n+1ell} the node $F$ uses one of $\ell_2,\ell_3$ to which it does not belong and then the other. Next we readily get that $F$ uses also the lines $\ell_4$ and $\ell_1.$ Thus the $4$-node line $\ell_1$ is used by $4$ nodes which, in view of Proposition  \ref{prp:4nl4} and Lemma \ref{lem:26664}, is a contradiction.
\end{proof}

\begin{lemma} \label{lem:33d} Any two triples of nodes corresponding to two distributions of type (c) or (d)
are disjoint.
\end{lemma}
\begin{proof}
Note that the variable lines with $2$ primary nodes in respective $\wti \ell$-m line sequences cannot be equal to any of the $4$-node line. Indeed the first five lines pass through all the nodes of the set $\Bset.$
Now if a sixth line becomes $4$-node line through a $2_m$-node of $\Bset$ then we have two lines passing through the $2_m$-node, which contradicts the remark $B2).$

Then since in each case of distributions (c) and (d) we have different $4$-node lines therefore the corresponding triples are disjoint.
 \end{proof}

Now assume that for at least two pairs of $4$-node lines we have the $\wti\ell$-m-d sequence (c). Then the two disjoint lines $\ell_2$ and $\ell_3$ in the respective $\wti\ell$-m line sequence are used by two triples of nodes, i.e., by six nodes.

In view of Proposition \ref{prp:66710} we get that each of the six nodes has either m-d sequence $(6,6,6,4,3,2)$ or $(6,6,5,5,3,2).$ The first case contradicts Lemma \ref{lem:26664}. While the second sequence clearly differs from the sequence (c), since there we have two invariable $4$-node lines.

Then assume that for one pair of $4$-node lines the  $\wti\ell$-m-d sequence (c) takes place and
for other $4$-node lines the sequence (d) takes place.

\noindent In view of the remarks $B3'), O1),$ and $O3),$ let us consider the following

Case 1) There are three $6$-node lines through $O$ used by three nodes and therefore there are at least seven $4$-node lines through $2_m$-nodes, and

Case 2) There are exactly two $6$-node lines through $O$ used by three nodes and  therefore there are at least six $4$-node lines through $2_m$-nodes.

Now recall that the two disjoint $6$-node lines are $\ell_2$ and $\ell_3.$ Denote the $6$-node lines passing through $O$ by $\ell_0,\ell_0',\ell_0''.$ Finally denote the above seven possible $4$-node lines through the $2_m$-nodes of $\Bset$ by $\alpha_i, i=1,\ldots,7.$

 In Case 1) we have six different triples.
Let us consider only the $4$-node and $6$-node lines in the respective line sequences:

 $(\alpha_1, \alpha_2, \ell_2, \ell_3);\quad (\alpha_3,\ell_2);\quad (\alpha_4, \ell_{3});\quad (\alpha_5, \ell_0);\quad (\alpha_6, \ell_0');\quad (\alpha_7, \ell_0'').$

  In Case 2) we have $5$ different triples corresponding to:

   $(\alpha_1, \alpha_2, \ell_2, \ell_3);\quad (\alpha_3, \ell_2);\quad (\alpha_4, \ell_{3});\quad
   (\alpha_5,  \ell_0);\quad (\alpha_6, \ell_0').$

 Note that in both cases all the possible $6$-node lines
are used by six nodes, counted also the triple usage of each of lines $\ell_0,\ell_0',\ell_0'',$ from $\Bset.$  Therefore, in view of Proposition \ref{prp:6810} and Lemma \ref{lem:26664}, no place for another $6$-node line in the line sequences, and consequently the additional triple usage. 

Thus all the lines with $5$ primary nodes in the line sequences, actually have to be exact $5$-node lines.

Let us show that these $5$-node lines are different. Indeed, assume conversely that a $5$-node line $\wti\ell$ is in two m-line sequences used by different triples. Then $\wti\ell$ is used by $6$ nodes. As we know, by Proposition \ref{prp:66710}, these nodes must have the m-d sequence  $(6,6,5,5,3,2),$ which clearly differs from (c) and (d).

Now in Case 1) we need for $16\ (=5\times 3+1)$ and in Case 2) we need for $13\ (=4\times 3+1)$ different $5$-node lines.

Below we show that actually there are not that many $3_\Aset\ 5$-node lines, which finishes the proof in this case.

For this end let us count the number of $2$-node lines in $\Bset.$  There are $9\ (=3\times 3)$ such lines through the three $2_m$-nodes (see Fig. \ref{case3}). Then there are $3$ such lines through the $1_m$ nodes.  Here we take into account that $C,C',C''\in \ell_2,$ and $D,D',D''\in\ell_3.$ Hence in all we may have atmost $12$ lines.

Finally, let us consider the case when for all $4$-node lines the $\wti\ell$-m-d sequence  (d) takes place.
Then in Case 1) we have seven disjoint triples whose union is $\Aset,$ which is a contradiction since $\#\Aset=18.$

In Case 2) we have $6$ different triples corresponding to:

$(\alpha_1, \ell_2);\quad (\alpha_2, \ell_2);\quad (\alpha_3, \ell_3);\quad (\alpha_4, \ell_3);\quad
   (\alpha_5,  \ell_0);\quad (\alpha_6, \ell_0').$

\noindent In this case no place for another $6$-node line too. Thus again all the lines with $5$ primary nodes actually are exact $5$-node lines.
Here we need for $18\ (=6\times 3)$ $5$-node lines. As we showed above there can be atmost $12$ such lines.


The work on the part of the first named author was carried out under grant 21T-A055 from the Scientific Committee of the Ministry of ESCS RA.



\begin{thebibliography}{99}
 \bibitem{dB07}
C.~de Boor, {Multivariate polynomial interpolation: conjectures
concerning GC sets}, Numer.\ Algorithms {\bf 45}  113--125 (2007).


\bibitem{B90}
J.~R.~Busch,
{A note on Lagrange interpolation in $\mathbb{R}^2$},
Rev.\ Un.\ Mat.\ Argentina {\bf 36}  33--38 (1990).


\bibitem{CG01}
J.~M.~Carnicer and M.~Gasca,
{A conjecture on multivariate polynomial interpolation},
Rev.\ R.~Acad.\ Cienc.\ Exactas F{\'i}s.\ Nat.\ (Esp.), Ser.~A
Mat.\ {\bf 95}  145--153 (2001).

\bibitem{CG02}
J.~M.~Carnicer and M.~Gasca,
On Chung and Yao's geometric characterization for  bivariate polynomial interpolation. In: \textit{Curve and Surface Design}, 21--30 (2002).


\bibitem{CY77}
K.~C.~Chung and T.~H.~Yao,
{On lattices admitting unique Lagrange interpolations},
SIAM J.\ Numer.\ Anal.\ {\bf 14} 735--743 (1977).

\bibitem{GM82}
M.~Gasca and J.~I.~Maeztu,
{On Lagrange and Hermite interpolation in $\mathbb{R}^k$},
Numer.\ Math.\ {\bf 39} 1--14 (1982).

\bibitem{HJZ09b}
H.~Hakopian, K.~Jetter, and G.~Zimmermann,
{A new proof of the Gasca-Maeztu conjecture for $n=4$},
J.\ Approx.\ Theory, {\bf 159} 224--242 (2009).

\bibitem{HJZ14}
H. Hakopian, K. Jetter, and G. Zimmermann, The Gasca-Maeztu
conjecture for $n=5$,  Numer. Math.,  {\bf 127} 685--713 (2014).

\bibitem{HK}
H. Hakopian and H. Kloyan, On the dimension of spaces of algebraic curves passing through $n$-indepent nodes. \textit{Proc. YSU. Phys. and Math. Sci.}, {\bf 53}, 3--13 (2019).

\bibitem{HKV}
H. Hakopian, H. Kloyan, and D. Voskanyan,
On plane algebraic curves passing through $n$-independent nodes,
J. Cont. Math. Anal., {\bf 56} 280–294 (2021).

\bibitem{HM}
H. Hakopian and A. Malinyan, Characterization of $n$-independent sets with no more than $3n$ points. \textit{Jaen J. Approx.}, {\bf 4}, 121--136 (2012).

\bibitem{HakTor}
H. Hakopian and S. Toroyan, On the Uniqueness of algebraic curves passing through  $n$-independent nodes. \textit{New York J. Math.}, {\bf 22} 441--452 (2016).


\bibitem{Raf}
L.~Rafayelyan,
{Poised nodes set constructions on algebraic curves},
East J.\ Approx.\ {\bf 17} 285--298 (2011).

\bibitem{GV1}
G. Vardanyan, On $n$-node lines in $GC_n$ sets, \textit{Proc. YSU. Phys. and Math. Sci.}, no. 1 , \textbf{55}, 1–12 (2021).

\bibitem{GV2}
G. Vardanyan, A new proof of the Gasca-Maeztu conjecture for n = 5,
J. Cont. Math. Anal., {\bf 57} 183–190 (2022).


\end{thebibliography}
\end{document}